\documentclass[a4paper,twoside,11pt]{amsart}

\usepackage{amssymb}
\usepackage{amsmath}
\usepackage{amsthm}
\usepackage{times}

\usepackage[dvipsnames]{xcolor}
\usepackage{geometry}
\usepackage{multicol}
\usepackage{ragged2e}
\usepackage{etoolbox}
\usepackage{mathtools}
\usepackage{tikz}

\theoremstyle{definition}
\newtheorem{definition}{Definition}
\newtheorem{example}{Example}
\newtheorem{corollary}{Corollary}
\newtheorem{lemma}{Lemma}

\newtheorem{theorem}{Theorem}

\newtheorem{remark}{Remark}

\usetikzlibrary{angles, quotes, intersections}

\definecolor{blue1}{RGB}{0, 0, 130}
\definecolor{blue2}{RGB}{0, 0, 200}
\definecolor{red2}{RGB}{180,0,0}
\definecolor{purp1}{RGB}{200,50,0}
\definecolor{green1}{RGB}{0,130,0}
\definecolor{lightgray}{RGB}{190,190,190}

\definecolor{mydarkblue}{rgb}{0,0.08,0.45}
\definecolor{mylessdarkblue}{rgb}{0,0.28,0.75}

\newcommand{\qrom}{{(\roman*)}}
\newcommand{\qalph}{{(\alph*)}}

\newtoggle{cheat}

\iftoggle{cheat}{
\geometry{
 a4paper,
 total={170mm,257mm},
 left=0.7in,
 top=0.7in,
 right=0.7in,
 bottom=0.7in,
 }
}{}

\newtoggle{midterm}
\toggletrue{midterm}

\newtoggle{defi}
\toggletrue{defi}

\iftoggle{cheat}{\togglefalse{defi}}{}

\newtoggle{key}
\toggletrue{key}

\newcommand{\vocab}[1]{{\color{blue2}\textit{{#1}}}}
\renewcommand{\vocab}[1]{{\emph{{#1}}}}

\newcommand{\paren}[1]{\left({#1}\right)}

\newcommand{\parenl}[1]{\left|{#1}\right|}

\newcommand{\RR}{\mathbb{R}}

\newtoggle{full}

\iftoggle{cheat}{\togglefalse{full}}{}

\newcommand{\disp}{\iftoggle{full}{\displaystyle}{}}

\newcommand{\eeps}{\epsilon}

\newcommand{\To}{\Rightarrow}
\newcommand{\From}{\Leftarrow}
\newcommand{\Iff}{\Longleftrightarrow}
\newcommand{\dsfrac}[2]{{\disp\frac{\disp{{#1}}}{\disp{{#2}}}}}

\newcommand{\norm}[1]{\|\myvec{#1}\|}

\usepackage{amsmath}
\usepackage{array}
\usepackage{enumitem}

\usepackage{multirow}
\usepackage{colortbl}
\usepackage{array}
\usepackage{wrapfig}
\usepackage{tabu}
\usepackage{makecell}
\usepackage{tabularx}

\newcolumntype{L}[1]{>{\raggedright\arraybackslash}p{#1}}
\newcolumntype{C}[1]{>{\centering\arraybackslash}p{#1}}
\newcolumntype{R}[1]{>{\raggedleft\arraybackslash}p{#1}}

\renewcommand{\iff}{\Leftrightarrow}

\newcommand{\normm}{\norm{\cdot}}

\usepackage{bm}

\usepackage{pgf,tikz}
\usepackage{mathrsfs}
\usepackage{multicol}
\usepackage{bigints}

\usepackage{graphicx}

\makeatletter
\newcommand\backnot{\mathrel{\mathpalette\back@not\relax}}
\newcommand\back@not[2]{%
  \makebox[0pt][l]{%
    \sbox\z@{$\m@th#1=$}%
    \kern\wd\z@
    \reflectbox{$\m@th#1\not$}%
  }%
}
\makeatother

\renewcommand{\norm}[1]{\left|\left|{#1}\right|\right|}

\newcommand{\sett}[1]{\left\{#1\right\}}

\usepackage[normalem]{ulem}

\newcommand{\ccal}{\mathcal}

\renewcommand{\sett}[1]{\left\{#1\right\}}

\renewcommand{\leq}{\leqslant}
\renewcommand{\geq}{\geqslant}

\usepackage{hyperref}

\newcommand{\KK}{\mathbb{K}}

\usepackage[T1]{fontenc}

\usepackage[
backend=biber,
style=numeric-comp,
]{biblatex}

\bibliography{outline.bib}

\title{Reformulating $\epsilon$-$\delta$ Limits in a Pedagogically Cleaner Way}
\author{Joel Q. L. Chang}
\date{}

\begin{document}
\maketitle

\begin{abstract}
We provide a simple reformulation of the $\epsilon$-$\delta$ limit definition introduced in undergraduate calculus courses that enhances its pedagogical value for conceptual understanding and computational skill.
\end{abstract}

\section{Introduction}

During the advent of calculus, the renowned scientists Newton and Leibniz independently developed differential and integral calculus using infinitesimals \cite{kutateladze2007excursus}.
These notions of infinitesimals were logically questionable until the 20th century, when mathematicians such as Cauchy and Weierstrass  \cite{kutateladze2007excursus} formulated the $\epsilon$-$\delta$ limit definition as follows.
\begin{definition}\label{def: vanilla_epsilon_delta}
	Let $c \in \RR$ and $A \subseteq \RR$ be a subset. We say that a function $f : A \to \RR$ has a \vocab{limit} $L$ at $c$, denoted $\lim_{x \to c}f(x) = L$, precisely when
	$$\forall \eeps > 0 \quad \exists \delta > 0 : \quad 0 < |x-c| < \delta \quad \To \quad |f(x) - L| < \eeps.$$
\end{definition}
Many introductory calculus courses include this classical $\epsilon$-$\delta$ limit definition as an introduction to undergraduate real analysis or as a means to rigorously discuss the concept of limits that are used to define continuity and differentiability \cite{Katz_2017}.

However, the $\epsilon$-$\delta$ limit definition has its fair share of critics citing poor pedagogical value. Furthermore, the notion of ``reverse-engineering'' the choice of $\delta$ given the goal of upper-bounding the output error $|f(x)-L|$ seems counterintuitive to many undergraduates.

In this paper, we use the affine transformation $x = c+t$ to reformulate the $\epsilon$-$\delta$ definition of the limit, and demonstrate its pedagogical value for proof-writing students. To our knowledge, we are not aware of any other source that adopts this reformulation, due to the pervasive adoption of Definition~\ref{def: vanilla_epsilon_delta}.

\section{The Reformulation}

We first state the definition of the special case $\lim_{t \to 0}f(t) = 0$:
\begin{definition}\label{def: zero_limit}
	Let $f$ be a real-valued function. Then
	$$\lim_{t \to 0} f(t) = 0\quad \Iff \quad \forall \eeps > 0 \quad \exists \delta > 0: \quad 0 < |t| < \delta \quad \To \quad |f(t)| < \eeps.$$
	We call $\delta$ the \vocab{input threshold} and $\eeps$ the \vocab{output threshold}. Furthermore,
	$$\lim_{t \to 0} f(t) = 0 \quad \Iff \quad \lim_{t \to 0} |f(t)| = 0$$
	since $||x|| = |x|$ for any $x \in \RR$. The latter form is useful when we want to compute limits in terms of upper-bounding. We illustrate some basic examples in our \href{https://youtu.be/103Qbhu0QlA}{YouTube video}.
\end{definition}
We can then formulate the usual definition of the limit in terms of this special case.
\begin{theorem}\label{thm: novel_epsilon_delta}
	Let $c \in \RR$, $A \subseteq \RR$, and $f : A \to \RR$ be a real-valued function. Then
	$$\lim_{x \to c} f(x) = L \quad \Iff \quad \lim_{t \to 0} (f(c+t)-L) = 0.$$
	Alternately phrased using the fizzle function vocabulary,
	$$\lim_{x \to c} f(x) = L \quad \Iff \quad \mbox{$f(c+\cdot)-L$ is a fizzle function}.$$
\end{theorem}
\begin{proof}[Proof of Theorem~\ref{thm: novel_epsilon_delta}]
	The substitution $x=c+t$ ensures
	$$0 < |x-c| < \delta \quad \Iff \quad 0 < |t| < \delta$$
	and
	$$|f(x)-L| < \eeps \quad \Iff \quad |f(c+t) - L| < \eeps.$$
\end{proof}
In the special case $c = 0$, there is no difference between the two definitions apart from a relabeling $x \to t$. Furthermore, any limit proved using Definition~\ref{def: vanilla_epsilon_delta} can in principle be proved via the equivalent formulation in Theorem~\ref{thm: novel_epsilon_delta} via the transformation $x=c+t$.

We therefore emphasise that the main advantage of this reformulation does not arise in mathematical novelty, but in pedagogical insight for students. This allows the heavy lifting of the theoretical $\eeps$-$\delta$ proofs to be carried out on the special case (Definition~\ref{def: zero_limit}), and derivation of the general limit theorems in terms of fizzle functions (see Section~\ref{sec: limit_theorems}).

Analytically, the transformation to $0 < |t| < \delta$ emphasises to the student that to prove $\lim_{t \to 0} f(t) = 0$, their two-fold goal  for any $\eeps > 0$ is to
\begin{enumerate}[label=\qalph]
	\item upper-bound the input error $|t|$ by their judiciously chosen input threshold $\delta$, so that they can
	\item upper-bound the output error $|f(t) - L|$ by the given output threshold $\epsilon$.
\end{enumerate}
And finally, to prove general statements of the form $\lim_{x \to c}f(x) = L$, they simply need to prove
$$\lim_{t \to 0} (f(c+t) - L) = 0$$
by Theorem~\ref{thm: novel_epsilon_delta}, which is familiar to them.
By letting the right-hand side of the limit to equal $0$, we have the added advantage of emphasising the goal of upper-bounding the error. Algebraically, this converts a problem involving factorisation into a problem involving expansion, which is easier though at the cost of tediousness. Furthermore, this notion is used to define asymptotes, which generalises the notion of
$$\lim_{x \to c} f(c) = L$$
when the right-hand side $L$ is not a constant.
\section{Common Pedogogical Examples}
To illustrate the point, we highlight several common examples of limits that students are taught to prove and demonstrate the veracity of the new formulation of the $\epsilon$-$\delta$ approach.
\begin{theorem}\label{thm: polynomial_cty}
	Let $p(x) = \sum_{k=0}^n a_k x^k$ be a nonzero polynomial. Then for any $c \in \RR$, $\lim_{x \to c}p(x) = p(c)$.
\end{theorem}
\begin{remark}
	Most exercises will provide $p(x)$ explicitly, for example, $p(x) = (x^2+3x)$. We prove this result in generality in Section~\ref{sec: appendix}, and demonstrate the two ways of concretely obtaining $\delta$.
\end{remark}
\begin{example}\label{eg: polynomial}
	Prove $\lim_{x \to 1} (x^2 +3x) = 4$ using the $\epsilon$-$\delta$ definition.
\end{example}
\begin{proof}[Solution to Example~\ref{eg: polynomial}]
	By Theorem~\ref{thm: novel_epsilon_delta}, we equivalently prove that $$\lim_{t \to 0} ({({(1+t)}^2 + 3(1+t))} - 4) = 0.$$
	By algebra,
	\begin{align*}
		({(1+t)}^2 + 3(1+t)-4)
		&= \cdots = t^2 + 5t.
	\end{align*}
	Hence it suffices to prove that
	$$\lim_{t \to 0} (t^2 + 5t) = 0.$$
	Fix $\eeps > 0$. We illustrate both choices of $\delta$.
	\begin{enumerate}[label=\qrom]
		\item Choose $\delta = \min\{\paren{\eeps/2}^{1/2}, {\paren{\eeps/10}}^{1/2}\}$. Then for the input error $0 < |t| < \delta$, we can bound the output error by
	\begin{align*}
	|t^2+5t|
	&\leq {|t|}^2 + 5|t|\\
	&<\delta^2 + 5\delta\\
	&\leq \frac{\eeps}{2} + \frac{\eeps}{2}
	= \eeps.
	\end{align*}
		\item Choose $\delta = \min \{1,\eeps/6\}$. Then
		\begin{align*}
		|t^2+5t|
		&< \delta^2 + 5\delta\\
		&= \delta(\delta + 5)\\
		&\leq \delta \cdot (1+5)\\
		&= 6\delta \leq 6 \cdot \frac{\eeps}{6} = \eeps.	
		\end{align*}
	\end{enumerate}
\end{proof}
\begin{theorem}\label{thm: rational_functions}
	Let $p(x), q(x)$ be nonzero polynomials. Let $c \in \RR$ such that $q(c) \neq 0$. Then
	$$\lim_{x \to c}\dsfrac{p(x)}{q(x)} = \dsfrac{p(c)}{q(c)}.$$
\end{theorem}
\begin{remark}
	Most exercises will provide $p(x),q(x)$ explicitly, for example, $p(x) = x+1, q(x) = (x-1)(x^2+1)$. We prove this result in generality in Section~\ref{sec: appendix}, and demonstrate the two ways of concretely how to obtain $\delta$.
\end{remark}
\begin{example}\label{eg: rational_limit}
	Prove ${\displaystyle \lim_{x \to 2} \dsfrac{x+1}{(x-1)(x^2+1)}} = \frac 35$ using the $\epsilon$-$\delta$ definition.
\end{example}
\begin{proof}[Solution to Example~\ref{eg: rational_limit}]
	By Theorem~\ref{thm: novel_epsilon_delta}, we equivalently prove that $$\lim_{t \to 0} \paren{\dsfrac{(2+t)+1}{((2+t)-1)({(2+t)}^2+1)}-\frac 25} = 0.$$
	By algebra,
	\begin{align*}
		\dsfrac{(2+t)+1}{((2+t)-1)({(2+t)}^2+1)}-\frac 35
		&= \cdots = \dsfrac{-22t-15t^2-3t^3}{5(1+t)({(2+t)}^2+1)}.
	\end{align*}
	Taking absolute signs,
	\begin{align*}
	\parenl{\dsfrac{(2+t)+1}{((2+t)-1)({(2+t)}^2+1)}-\frac 35}
	&= \frac 15 \cdot \frac 1{|1+t|} \cdot \frac{1}{{(2+t)}^2+1} \cdot (22+15|t| + 3{|t|}^2) \cdot |t|
	\end{align*}
	Fix $\eeps > 0$. Choose $\delta = \min\{1/2, \eeps/4\}$.
	Then for $0< |t| < \delta$, the reverse triangle inequality yields
	$$|1+t| \geq 1 - |t| > 1-\delta \geq 1 - \frac 12 = \frac 12 \quad \To \quad \dsfrac{1}{|1+t|} < \frac 12.$$
	For any $t \in \RR$, $${(2+t)}^2 + 1 \geq 0 + 1 = 1\quad \To \quad \dsfrac{1}{{(2+t)}^2 + 1} \leq 1.$$
	Furthermore, $$|t| < \delta \leq \frac 12 < 1 \quad \To \quad (22 + 15|t| + 3{|t|}^2) < 22 + 15 \cdot 1 + 3 \cdot 1^2 = 40.$$
	Hence,
	\begin{align*}
	\parenl{\dsfrac{(2+t)+1}{((2+t)-1)({(2+t)}^2+1)}-\frac 35}
	&= \frac 15 \cdot \frac 1{|1+t|} \cdot \frac{1}{{(2+t)}^2+1} \cdot (22+15|t| + 3{|t|}^2) \cdot |t|\\
	&< \frac 15 \cdot \frac 1{2} \cdot 1 \cdot 40 \cdot \frac{\eeps}{4} = \eeps.
	\end{align*}
\end{proof}
\section{Implementing the Limit Theorems}
\label{sec: limit_theorems}
This centre-at-zero approach can simplify many convoluted proofs, while maintaining the essence of why they work. In fact, this is arguably the last section (apart from the supplementary Section~\ref{sec: appendix}) where we employ the $\eeps$-$\delta$ definition of a limit, since the other limit theorems are essentially algebraic derivations of these results.
\begin{theorem}\label{thm: baby_example}
	$\lim_{t \to 0} t = 0$.
\end{theorem}
\begin{proof}[Proof of Theorem~\ref{thm: baby_example}] Fix $\eeps > 0$. Choose $\delta = \eeps$. Then
	$$0 < |t| < \delta \quad \To \quad |t| < \delta \leq \eeps.$$
\end{proof}
Consider the limit of a sum of functions.
\begin{theorem}\label{thm: limit_of_sum}
	Let $f,g$ be real-valued functions. If $\lim_{t \to 0}f(t) = \lim_{t \to 0} g(t) = 0$, then $$\lim_{t \to 0} (f+g)(t) = 0.$$
\end{theorem}
\begin{proof}[Proof of Theorem~\ref{thm: limit_of_sum}]
	We want to show that 
	$$\lim_{t \to 0} (f+g)(t) = 0.$$
	By the triangle inequality,
	\begin{align*}
		|(f+g)(t)| & \leq |f(t)| + |g(t)|.
	\end{align*}
	Fix $\eeps > 0$.
	Our goal is simply to upper bound both summands by $< \eeps/2$ each. For the real number $\eeps/2 > 0$, since $\lim_{t \to 0}f(t)=0$, there exists some $\delta_1 > 0$ such that
	$$0 < |t| < \delta_1 \quad \To \quad |f(t)| < \frac{\eeps}{2}.$$
	For the real number $\eeps/2 > 0$, since $\lim_{t \to 0}g(c+t)=0$, there exists some $\delta_2 > 0$ such that
	$$0 < |t| < \delta_2 \quad \To \quad |g(t)| < \frac{\eeps}{2}.$$
	Choose $\delta = \min\{\delta_1,\delta_2\}$. Then for $0<|t| <\delta$,
	\begin{align*}
		|(f+g)(t)| 
		&\leq |f(t)| + |g(t)| < \frac{\eeps}{2} + \frac{\eeps}{2} = \eeps.
	\end{align*}
\end{proof}
We can similar extend this notion to show that taking limits is linear, and even allow students to prove it as a benign exercise.
\begin{theorem}\label{thm: limit_of_scalar_mult}
	Let $k \in \RR$ and $f$ be a real-valued function. If $\lim_{t \to 0}f(t) = 0$, then $$\lim_{t \to 0} (kf)(t) = 0.$$
\end{theorem}
We can even show that $\lim_{t \to 0}(\cdot)$ is multiplicative.
\begin{theorem}\label{thm: multiplicativity_of_lim}
	Let $f,g$ be real-valued functions. If $\lim_{t \to 0}f(t) = \lim_{t \to 0} g(t) = 0$, then $$\lim_{t \to 0} (f\cdot g)(t) = 0.$$
\end{theorem}
\begin{proof}[Proof of Theorem~\ref{thm: multiplicativity_of_lim}]
	We want to show that 
	$$\lim_{t \to 0} (f\cdot g)(t) = 0.$$
	By algebra,
	\begin{align*}
		|(f\cdot g)(t)|
		&= |f(t) \cdot g(t)| = |f(t)| \cdot |g(t)|.
	\end{align*}
	Since
	$$\lim_{t \to 0} f(t) = 0, \quad \lim_{t \to 0} g(t) = 0,$$ there exist $\delta_1,\delta_2>0$ such that
	\begin{align*}
		0 < |t| < \delta_1 \quad &\To\quad |f(t)| < \sqrt{\eeps},\\
		0 < |t| < \delta_2 \quad &\To\quad |g(t)| < \sqrt{\eeps}.
	\end{align*}
	Choose $\delta = \min\{\delta_1,\delta_2\}$. Then
	\begin{align*}
		|(f\cdot g)(t)|
		&=  |f(t)| |g(t)| \leq \sqrt{\eeps}\cdot \sqrt{\eeps} = \eeps.
	\end{align*}
\end{proof}
By Theorem~\ref{thm: novel_epsilon_delta}, we can rewrite these theorems to centre around the point $c \in \RR$ rather than just $0$.
In particular, the proof for the continuity of $\sin(\cdot)$ at $0$ becomes strikingly obvious.
\begin{theorem}\label{thm: cty_of_sin_at_zero}
	$\lim_{t \to 0} \sin(t) = 0$.
\end{theorem}
\begin{proof}[Proof of Theorem~\ref{thm: cty_of_sin_at_zero}]
	By a geometric argument, for $0 < |t| < \frac 12 \pi$,
	$$|{\sin(t)}| < |t|.$$
	Fix $\eeps > 0$.
	Choose $\delta = \min\{\frac 12 \pi, \eeps\}$. Then for $0 < |t| < \eeps$,
	$$|{\sin(t)}| < \eeps.$$
\end{proof}
For limits of quotients, we need to be careful, since we cannot allow for zero-denominator errors. Thus, we will first consider the denominator having limit $1$, before generalising using our previous rules.
\begin{theorem}\label{thm: limit_of_quotient}
	Let $f$ be a real-valued function. Suppose $\lim_{t \to 0}f(t) = 1$. Then $\lim_{t \to 0} \dsfrac{1}{f(t)} = 1$.
\end{theorem}
\begin{proof}[Proof of Theorem~\ref{thm: limit_of_quotient}]
	By algebra,
	$$\parenl{\frac{1}{f(t)} - 1} = \frac{|f(t)-1|}{|f(t)|}.$$
	Since $\lim_{t \to 0} (f(t) - 1) = 0$, for the real number $1/2 > 0$, there exists $\delta_1 > 0$ such that for $0 < |t| < \delta_1$,
	\begin{align*}
	|f(t)-1|<1/2\quad &\To \quad -1/2 < f(t) - 1 < 1/2\\
	\quad &\To \quad f(t) > 1-1/2 = 1/2\\
	&\To \quad \dsfrac{1}{|f(t)|} = \frac{1}{f(t)} < \frac 12.	
	\end{align*}
	Fix $\eeps > 0$. Since $\lim_{t \to 0} (f(t) - 1) = 0$, for the real number $2\eeps > 0$, there exists $\delta_2 > 0$ such that for $0 < |t| < \delta_1$,
	$$|f(t) - 1| < 2\eeps.$$
	Choose $\delta = \min\{\delta_1,\delta_2\}$. Then
	\begin{align*}
	\parenl{\frac{1}{f(t)} - 1} 
	&= \frac{1}{|f(t)|}\cdot |f(t)-1| < \frac 12 \cdot 2\eeps = \eeps.
	\end{align*}
\end{proof}
The squeeze theorems also receives a more elegant proof.
\begin{theorem}\label{thm: squeeze_theorem}
	Let $f,g,h$ be real-valued functions. Suppose $\lim_{t \to 0} f(t) = \lim_{t \to 0} h(t) = 0$ and $f(t) \leq g(t) \leq h(t)$ for $t$ near $0$. Then $$\lim_{t \to 0} g(t) = 0.$$
\end{theorem}
\begin{proof}[Proof of Theorem~\ref{thm: squeeze_theorem}]
	Since $f(t) \leq g(t) \leq h(t)$ for $t$ near $0$, there exists $\delta_1 > 0$ such that
	$$0 < |t| < \delta_1 \quad \To \quad f(t) \leq g(t) \leq h(t).$$
	Fix $\eeps > 0$. We want to find $\delta > 0$ such that
	$$0 < |t| <\delta \quad \To \quad |g(x)| < \eeps \quad \Iff \quad -\eeps < g(x) < \eeps.$$
	Since $\lim_{t \to 0} f(t) = L$, for the real number $\eeps > 0$, there exists $\delta_2 > 0$ such that
	$$0 < |t| < \delta_1 \quad \To \quad |f(t)| < \eeps/2 \quad \To \quad f(t) > -\eeps/2.$$
	Since $\lim_{t \to 0} f(t) = L$, for the real number $\eeps > 0$, there exists $\delta_3 > 0$ such that
	$$0 < |t| < \delta_2 \quad \To \quad |h(t)| < \eeps/2 \quad \To \quad h(t) < \eeps/2.$$
	Choose $\delta = \min\{\delta_1,\delta_2, \delta_3\}$. Then
	\begin{align*}
	0 < |t| < \delta \quad 
	&\quad \To  \quad -\eeps < f(t) \leq g(t) \leq h(t) < \eeps\\
	&\quad \To \quad -\eeps < g(t) < \eeps\\
	&\quad \To \quad |g(t)| < \eeps,	
	\end{align*}
	as required.
\end{proof}
\section{Generalising the Limit Theorems}
\label{sec: gen_limit_theorems}
We can now generalise these results for nonzero limits. In fact, since the $\epsilon$-$\delta$ heavy lifting has been done in Section~\ref{sec: limit_theorems}, by assuming the results therein, we reduce their proofs to relatively benign high-school computational exercises.
\begin{corollary}\label{cor: baby_example}
	For any $c \in \RR$, $\lim_{x \to c} x = c$.
\end{corollary}
\begin{proof}
	By Theorem~\ref{thm: novel_epsilon_delta}, we equivalently want to show that
	$$\lim_{t \to 0} ((c+t)-c) = 0 \quad \Iff \quad \lim_{t \to 0} t = 0,$$
	which holds by Theorem~\ref{thm: baby_example}.
\end{proof}
We can furthermore generalise the limit theorems for nonzero limit points $c$.
\begin{corollary}\label{cor: general_limit_point}
	Let $f,g, h$ be real valued functions, $c \in \RR$, and $k \in \RR$. Suppose $$\lim_{x \to c} f(x) = \lim_{x \to c}g(x) = 0 \quad \mbox{and} \quad \lim_{x \to c}h(x) = 1.$$ Then
	\begin{align*}
		\lim_{x \to c} (f+ g)(x) &= 0,\\
		\lim_{x \to c} (kf)(x) &= 0,\\
		\lim_{x \to c} (f\cdot  g)(x) &= 0,\\
		\lim_{x \to c} \dsfrac{1}{h(x)} = 1.
	\end{align*}
	Furthermore, let $p$ be a real-valued function such that $f(x) \leq p(x) \leq g(x)$ for $x$ near $c$. Then
	$$\lim_{x \to c} p(x) = 0.$$
\end{corollary}
\begin{proof}[Proof of Corollary~\ref{cor: general_limit_point}]
	By Theorem~\ref{thm: novel_epsilon_delta}, the hypotheses are equivalent to
	$$\lim_{t \to 0} f(c+t) = \lim_{t \to 0} g(c+t) = 0, \quad \lim_{t \to 0} h(c+t) = 1.$$
	Therefore,
	\begin{align*}
		\lim_{x \to c} (f+ g)(x)
		&= \lim_{t \to 0} (f+ g)(c+t) = 0,\\
		\lim_{x \to c} (kf)(x) &= \lim_{t \to 0} (kf)(c+t) = 0,\\
		\lim_{x \to c} (f\cdot  g)(x) &= \lim_{t \to 0} (f\cdot  g)(c+t) = 0,\\
		\lim_{x \to c} \dsfrac{1}{h(x)} &= \lim_{t \to 0} \dsfrac{1}{h(c+t)} = 1.
	\end{align*}
	For the last result (i.e. the analog of Theorem~\ref{thm: squeeze_theorem}), $x = c+t$ being near $c$ is equivalent to $t$ being near $0$, and $f(x) \leq p(x) \leq g(x)$ is translated to
	$$f(c+t) \leq p(c+t) \leq g(c+t).$$
	By Theorem~\ref{thm: squeeze_theorem}, $\lim_{x \to c} p(x) = \lim_{t \to 0} p(c+t) = 0$.
\end{proof}
\begin{theorem}\label{thm: generalise}
	Let $f,g$ be real-valued functions, $c \in \RR$, and $k \in \RR$. Suppose $\lim_{x \to c} f(x) = L$ and $\lim_{x \to c} g(x) = M$. Then
	\begin{align*}
		\lim_{x \to c} (f+ g)(x) &= L+M,\\
		\lim_{x \to c} (kf)(x) &= kL,\\
		\lim_{x \to c} (f- g)(x) &= L-M,\\
		\lim_{x \to c} (f\cdot  g)(x) &= LM.
	\end{align*}
	Furthermore, let $p$ be a real-valued function such that $f(x) \leq p(x) \leq g(x)$ for $x$ near $c$. Suppose $M = L$. Then
	$$\lim_{x \to c} p(x) = L.$$
\end{theorem}
\begin{proof}
	By Definition~\ref{def: vanilla_epsilon_delta}, we are equivalently given that
	$$\lim_{x \to c} (f(x)-L) = \lim_{x \to c} (g(x)-M) = 0.$$
	By Corollary~\ref{cor: general_limit_point},
	$$\lim_{x \to c} ((f+g)(x)-(L+M)) = \lim_{x \to c} ((f(x)-L) + (g(x)-M)) = 0,$$
	which is equivalent to $$\lim_{x \to c} (f+ g)(x) = L+M.$$
	By Corollary~\ref{cor: general_limit_point},
	$$\lim_{x \to c} ((kf)(x) - kL) = \lim_{x \to c} (k(f(x)-L)) = 0,$$
	which is equivalent to
	$$\lim_{x \to c} (kf)(x) = kL.$$
	Writing $f-g = f+(-1)g$,
	\begin{align*}
		\lim_{x \to c} (f- g)(x)
		&= \lim_{x \to c} (f+(-1)g)(x)\\
		&= \lim_{x \to c} f(x) + \lim_{x \to c} (-1) g(x)\\
		&= \lim_{x \to c} f(x) + (-1) \lim_{x \to c} g(x)\\
		&= \lim_{x \to c} f(x) - \lim_{x \to c} g(x) = L-M.
	\end{align*}
	By Corollary~\ref{cor: general_limit_point},
	$$0 = \lim_{x \to c} ((f(x) - L)(g(x) - M)) = \lim_{x \to c} (f(x)g(x) - M f(x) - L g(x) + LM).$$
	By the linearity of $\lim_{x \to c}(\cdot)$,
	\begin{align*}
		\lim_{x \to c} (f(x)g(x))
		&= \lim_{x \to c} ((f(x)g(x) - M f(x) - L g(x) + LM) + (Mf(x) + Lg(x) - LM))\\
		&= \lim_{x \to c} (f(x)g(x) - M f(x) - L g(x) + LM) + M \cdot \lim_{x \to c} f(x) + L \cdot \lim_{x \to c} g(x) - LM\\
		&= 0 + ML + LM -LM = ML.
	\end{align*}
	Finally, by Corollary~\ref{cor: general_limit_point}, $f(x) - L \leq p(x) - L \leq g(x) - L$ for $x$ near $c$ allows for
	$$\lim_{x \to c} (p(x) - L) = 0 \quad \Iff \quad \lim_{x \to c}p(x) = L.$$
\end{proof}
\begin{remark}
	The conventional $\eeps$-$\delta$ proofs require the obscure step of creating variables and having to reverse-engineer the choice of $\delta$. While this has value in and of itself, it may be an unnecessarily complicated and arguably extraneous barrier-to-entry to clearly communicating the $\eeps$-$\delta$ approach to proving limit statements. In this formulation, the mental gymnastics are essentially ``covered'' or ``hidden'' by the limit theorems of $\lim_{t \to 0}(\cdot)$ and do not appear as prominently in the proofs.
\end{remark}
\begin{corollary}
	By Theorems~\ref{thm: baby_example},~\ref{thm: limit_of_sum}, and~\ref{thm: multiplicativity_of_lim}, for any polynomial $p(x) = \sum_{k=0}^n a_k x^k$, $\lim_{x \to c} p(x) = p(c)$.
\end{corollary}
Furthermore, the proof for the continuity of $\sin(\cdot)$ reduces to merely a matter of trigonometric manipulation.
\begin{corollary}\label{cor: cty_of_sin}
	For $c \in \RR$, $\lim_{x \to c} \sin(x) = \sin(c)$.
\end{corollary}
\begin{proof}[Proof of Corollary~\ref{cor: cty_of_sin}]
	For general $c \in \RR$, trigonometric formulae yield
	\begin{align*}
		\sin(c+t) -\sin(c)
		&= \sin(c) \cos(t) + \cos(c) \sin(t) - \sin(c)\\
		&= -\sin(c) (1-\cos(t)) + \cos(c) \sin(t)\\
		&= -\sin(c) \cdot 2 \sin^2(t/2) + \cos(c) \sin(t).
	\end{align*}
	Applying $\lim_{t \to 0}(\cdot)$ on both sides, and applying the linearity and multiplicativity of $\lim_{t \to 0}(\cdot)$,
	\begin{align*}
		\lim_{t \to 0}(\sin(c+t) -\sin(c))
		&= \lim_{t \to 0}(-{\sin(c)} \cdot 2 \sin^2(t/2) + \cos(c) \sin(t))\\
		&= -2\sin(c) \cdot \lim_{t \to 0}  \sin^2(t/2) + \cos(c) \cdot \lim_{t \to 0} \sin(t)\\
		&= -2\sin(c) \cdot 0^2 + \cos(c) \cdot 0 = 0.
	\end{align*}
\end{proof}
\begin{corollary}\label{cor: limit_of_quotient}
	Let $f,g$ be real-valued functions. Suppose $\lim_{x \to c}f(x) = L$, $\lim_{x \to c} g(x) = M \neq 0$. Then $\lim_{x \to c} (f/g)(x) = L/M$.
\end{corollary}
\begin{proof}[Proof of Corollary~\ref{cor: limit_of_quotient}]
	By Theorem~\ref{thm: generalise},
	$$\lim_{x \to c} \frac {g(x)}M = \frac 1M \cdot M = 1.$$
	By Corollary~\ref{cor: general_limit_point},
	$$\lim_{x \to c} \frac{M}{g(x)} = 1.$$
	By Theorem~\ref{thm: generalise} again,
	$$\lim_{x \to c} \frac{1}{g(x)} = \lim_{x \to c} \frac 1M \cdot \frac{M}{g(x)} = \frac 1M \cdot 1 = \frac 1M.$$
	By Theorem~\ref{thm: generalise},
	$$\lim_{x \to c} (f/g)(x)=\lim_{x \to c} \frac{f(x)}{g(x)} = \lim_{x \to c} f(x) \cdot \frac{1}{g(x)} = L \cdot \frac 1M = \frac LM.$$
\end{proof}
\section{Reformulating Infinity-Related Limits}
The conventional definition of limits at infinity is as follows.
\begin{definition}\label{def: limits_at_infinity}
	Let $A \subseteq \RR$ be a subset. For any function $f : A \to \RR$, we write
	$$\lim_{x \to \infty} f(x) = L \quad \Iff \quad \forall \eeps > 0 \quad \exists M > 0:\quad x > M \quad \To \quad |f(x)-L|<\eeps.$$
\end{definition}
We remark its equivalent form and demonstrate the usefulness of the latter version.
\begin{theorem}\label{thm: equiv_limits_at_infinity}
	Let $A \subseteq \RR$ be a subset. Then
	$$\lim_{x \to \infty} f(x) = L \quad \Iff \quad \lim_{t \to 0^+} (f(1/t) - L) = 0.$$
\end{theorem}
\begin{proof}[Proof of Theorem~\ref{thm: equiv_limits_at_infinity}]
	Using $x = 1/t$, set $\delta = 1/M > 0$ and the proof follows the equivalence of
	$$x > M > 0 \quad \Iff \quad 0 < t < \delta.$$
\end{proof}
This allows us to derive the various limit theorems for when $x \to \infty$ in terms of the formulation in Theorem~\ref{thm: novel_epsilon_delta}, since it follows from the vanilla limit theorems at finite $c \in \RR$.
We illustrate the usefulness of the alternate formulation below.
\begin{example}\label{eg: limits_at_infinity}
	Prove that ${\displaystyle \lim_{x \to \infty} \dsfrac{2x+1}{3x-4}} = \frac 23$.
\end{example}
\begin{proof}[Proof of Example~\ref{eg: limits_at_infinity}]
	By Theorem~\ref{thm: equiv_limits_at_infinity}, we equivalently prove that ${\displaystyle \lim_{t \to 0^+} \paren{\dsfrac{2(1/t)+1}{3(1/t)-4}-\frac 23}} = 0$.
	By algebra,
	\begin{align*}
		\frac{2(1/t)+1}{3(1/t)-4} - \frac 23
		&= \cdots = \frac{7t}{3(3-4t)}.
	\end{align*}
	Taking absolute signs,
	\begin{align*}
		\parenl{\frac{2(1/t)+1}{3(1/t)-4} - \frac 23}
		&= \frac 73 \cdot \frac 1{|3-4t|} \cdot |t|.
	\end{align*}
	Choose $\delta = \min\{3/8, 9\eeps/14\}$. Then for $0 < t < \delta$,
	$$0 < \frac 32 = 3 - 4 \cdot \frac 38 =3-4\delta <3-4t < 3 \quad \To \quad \frac 1{|3-4t|} = \frac 1{3-4t} < \frac 23.$$
	Hence,
	\begin{align*}
		\parenl{\frac{2(1/t)+1}{3(1/t)-4} - \frac 23}
		&= \frac 73 \cdot \frac 1{|3-4t|} \cdot |t|\\
		&< \frac 73 \cdot \frac 23 \cdot \frac 9{14} \eeps = \eeps.
	\end{align*}
\end{proof}

\section{Further Applications}

Since the \vocab{derivative} $f'$ of a function $f$ is essentially defined as a limit, the above-mentioned techniques apply, since
$$f'(c) = \lim_{x \to c} \frac{f(x)-f(c)}{x-c} \quad \Iff \quad \lim_{t \to 0} \paren{\frac{f(c+t)-f(c) - f'(c)t}{t}} = 0.$$
In fact, one uses the equivalent formulation to define \vocab{Frech\'et} derivatives, in the setting of normed spaces.
\begin{definition}
	Let $V$ be a vector space over $\RR$. A function $\normm$ is a \vocab{norm} if it satisfies the four conditions:
	\begin{enumerate}
		\item \vocab{Non-negativity} For any $v \in V$, $\norm{v} \geq 0$.
		\item \vocab{Non-degeneracy} For any $v \in V$, $\norm{v} = 0 \iff v = 0$.
		\item \vocab{Homogeneity} For any $v \in V, c \in \KK$, $\norm{cv} = |c| \norm{v}$.
		\item \vocab{Triangle inequality} For any $u,v \in V$, $\norm{u+v} \leq \norm{u} + \norm{v}$.
	\end{enumerate}
	For instance, the absolute value function $\normm := |\cdot|$ is a norm on $X:=\RR$.
\end{definition}
\begin{definition}
	Let $f : V \to W$ be a map between normed spaces, denoting the norms in $V,W$ by $\normm_V,\normm_W$ respectively. For any open set $U \subset V$, the \vocab{Frech\'et} derivative at $v \in U$ is the unique map $f'(v) : V \to W$ such that
	$$\lim_{\norm{u}_V \to 0} \frac{\norm{f(v+u) - f(v) - (f'(v))(u)}_W}{\norm{u}_V} = 0.$$
\end{definition}
This limit notion also generalise nicely to topological vector spaces, which are itself generalisations of normed spaces.
\begin{definition}
	A \vocab{topological vector space} is a vector space $V$ over a topological field $\KK$ equipped with a topology $\ccal T$ with the property that the addition and scalar multiplication maps $$+ : V \times V \to V, (u,v) \mapsto u+v, \quad \cdot : \KK \times V \to V, (c,v) \mapsto c \cdot v$$
	respectively are continuous. 	
\end{definition}
	Let $f : X \to Y$ be a map between topological vector spaces. 
	Since topological vector spaces are completely regular \cite{Robdera+2022+246+254}, they are Hausdorff, and limits therein are unique. We can then define limits as follows.
\begin{definition}
Denote $\lim_{t \to 0} f(t) = 0$ if for any neighbourhood $V$ of $0 \in Y$, there exists a neighbourhood $U$ of $0 \in X$ such that
	$$f(U) \subseteq V.$$
	Finally, we generalise to the limit at $c \in X$ using continuity of addition and scalar multiplication by implementing
	$$\lim_{x \to c}f(x) = L \quad \iff \quad \lim_{t \to 0}(f(c+t) - L) = 0$$
	where we verify the right-hand side via the function $g = f(c + \cdot) - L$.
\end{definition}
We perform a quick check that this definition agrees with the conventional limit definition.
\begin{theorem}
	The following equivalence holds:
	$$\lim_{x \to c}f(x) = L \quad \iff \quad \text{$\forall_{V \ni L} \quad \exists_{U \ni c}:\quad f(U) \subseteq V$}.$$
\end{theorem}
\begin{proof}
	For $(\To)$, fix a neighbourhood $V$ of $L$. By translation invariance, the set $V-L := \{v - L : v \in V\}$ is a neighbourhood of $0$. By the definition of the left-hand side, there exists a neighbourhood $U$ of $0$ such that
	$$g(U) \subseteq V- L.$$
	The set $U + c := \{u + c : u \in U\}$ is a neighbourhood of $c$. Then one checks by linear algebra that
	$$f(U+c) = g(U) + L \in V,$$
	as required.
	For $(\From)$, fix a neighbourhood $V$ of $0 \in Y$. The set $V + L$ is a neighbourhood of $L$. By the right-hand side, there exists a neighbourhood $U$ of $c$ such that
	$$f(U) \subseteq V+L.$$
	One can then check that $U - c$ is a neighbourhood of $0 \in X$, and that
	$$g(U-c) = f(U) - L \subseteq V,$$
	as required.
\end{proof}
\section{Conclusion}
The simple modification of defining limits in terms of the origin for vector spaces massively simplifies computational complications for undergraduate students and novices in real analysis, who are acclimatising to the required rigour. Furthermore, the core notions of neighbourhoods are sufficiently general to account for abstract topological vector spaces, which inevitably simplifies further computations. We hope this resource will edify university students of any mastery of real analysis and topology in their mathematical journey.

\printbibliography

@misc{kutateladze2007excursus,
      title={Excursus into the History of Calculus}, 
      author={S. Kutateladze},
      year={2007},
      eprint={math/0701068},
      archivePrefix={arXiv},
      primaryClass={math.HO}
}

@article{Katz_2017,
	doi = {10.5642/jhummath.201701.07},
  
	url = {https://doi.org/10.5642%2Fjhummath.201701.07},
  
	year = 2017,
	month = {jan},
  
	publisher = {Claremont Colleges Library},
  
	volume = {7},
  
	number = {1},
  
	pages = {87--104},
  
	author = {Mikhail Katz and Luie Polev},
  
	title = {From Pythagoreans and Weierstrassians to True Infinitesimal Calculus},
  
	journal = {Journal of Humanistic Mathematics}
}

@article{Robdera+2022+246+254,
url = {https://doi.org/10.1515/taa-2022-0131},
title = {Compactness Principles for Topological Vector Spaces},
title = {},
author = {Mangatiana A. Robdera},
pages = {246--254},
volume = {10},
number = {1},
journal = {Topological Algebra and its Applications},
year = {2022},
lastchecked = {2023-07-30}
}

\section{Appendix}
\label{sec: appendix}

In this section, we prove theorems whose proofs are supplementary to the main material. These are well-establish real analytic results in the mathematical community that are included for the sake of completeness.
\begin{proof}[Proof of Theorem~\ref{thm: limit_of_scalar_mult}]
	We want to show that 
	$$\lim_{t \to 0} (kf)(t) = 0.$$
	Suppose $k \neq 0$. By algebra,
	\begin{align*}
		|(kf)(t)| &= |k\cdot f(t)| \leq |k| |f(t)|.
	\end{align*}
	Fix $\eeps > 0$.
	For the real number $\eeps/|k| > 0$, since $\lim_{t \to 0}f(t)=0$, there exists some $\delta > 0$ such that
	$$0 < |t| < \delta \quad \To \quad |f(t)| < \frac{\eeps}{|k|}.$$
	Then for $0<|t| <\delta$,
	\begin{align*}
		|(kf)(t)|
		& = |k| |f(t)| \leq |k| \cdot \frac{\eeps}{|k|} = \eeps.
	\end{align*}
\end{proof}
\begin{proof}[Proof of Theorem~\ref{thm: polynomial_cty}]
	By Theorem~\ref{thm: novel_epsilon_delta} it suffices to show that $\lim_{t \to 0} (p(c+t) - p(c))=0$. Equivalently, we want to prove the quantified statement
	$$\forall \eeps > 0 \quad \exists \delta > 0: \quad 0<|t| < \delta \quad \To \quad |p(c+t) -p(c)| < \eeps.$$
	Fix $\eeps > 0$. Choose $\delta = \square$, where $\square > 0$ is to be determined from the analysis of our calculations. Assume $0 < |t| < \delta$. Then
	\begin{align*}
	p(c+t) - p(c)
	&= \sum_{k=0}^n a_k {(c+t)}^k - \sum_{k=0}^n a_k {c}^k\\
	&= \sum_{k=0}^n a_k ({(c+t)}^k - c^k).
	\end{align*}
	By the vanilla binomial theorem,
	$${(c+t)}^k - c^k = \sum_{j=0}^k {k \choose j} c^{k-j} t^j - c^k = \sum_{j=1}^k {k \choose j} c^{k-j} t^j.$$
	This simplifies $p(c+t) - p(c)$ to
	\begin{align*}
		p(c+t) - p(c)
		&= \sum_{k=0}^n a_k \paren{\sum_{j=1}^k {k \choose j} c^{k-j} t^j}
		= \sum_{k=0}^n \sum_{j=1}^k a_k {k \choose j} c^{k-j} t^j\\
		&= \sum_{j=1}^n \sum_{k=j}^n a_k {k \choose j} c^{k-j} t^j
		= \sum_{j=1}^n b_j t^j,
	\end{align*}
	where $$b_j = \sum_{k=j}^n a_k {k \choose j} c^{k-j}$$
	can be computed with relatively simple algebraic manipulation.
	Applying $|\cdot|$ and using the triangle inequality, we can bound the output error by
	\begin{align*}
		|p(c+t)-p(c)|
		&= \parenl{\sum_{j=1}^n b_j t^j}
		\leq \sum_{j=1}^n |b_j| {|t|}^j
		< \sum_{j=1}^n |b_j| {\square}^j.
	\end{align*}
	There are two ways to determine the choices for $\square$. We illustrate them below.
	\begin{enumerate}[label=\qrom]
		\item To allow the right-hand side to be bounded above by $\eeps$, simply set each $|b_j| \square^j \leq \eeps/n$. That is, stipulate
	$$\square \leq \paren{\frac{\eeps}{n|b_j|}}^{\frac 1j}, \quad b_j \neq 0,$$
	so that by the triangle inequality, 
	$$|p(c+t) - p(c)| < \sum_{j=1}^n |b_j| \square^j \leq \sum_{j=1}^n \frac{\eeps}{n} = \eeps.$$
	Thus, the choice $\square = \min\sett{{(\eeps/(n|b_j|))}^{1/j} : b_j \neq 0}$ works.
		\item Alternatively, one could factor the $|t|$ to get
		$$|p(c+t)-p(c)| \leq |t| \sum_{j=1}^n |b_j| {|t|}^{j-1} < \square \sum_{j=1}^n |b_j| {\square}^{j-1}.$$
		By first stipulating $\square \leq 1$, one could upper bound the sum by
		$$\sum_{j=1}^n |b_j| {\square}^{j-1} \leq \sum_{j=1}^n |b_j|,$$
		which is positive since $p(x)$ is a nonzero polynomial. Further stipulating $\square \leq \eeps / \sum_{j=1}^n |b_j|$, one can obtain
		$$|p(c+t)-p(c)| < \square \sum_{j=1}^n |b_j| {\square}^{j-1} \leq  \paren{\frac{\eeps}{\sum_{j=1}^n |b_j|}} \cdot \sum_{j=1}^n |b_j| = \eeps.$$
		Thus, the choice $\square = \min \sett{1,  \eeps / \sum_{j=1}^n |b_j|}$ works.
	\end{enumerate}
\end{proof}
\begin{lemma}\label{thm: continuity_positivity}
	For any nonzero polynomial $p(x)$ with $p(c) \neq 0$, there exists $r > 0$ and $\delta > 0$ such that for $0 < |t| < \delta$, $|p(c+t)| > r$
\end{lemma}
\begin{proof}
	Since $\lim_{t \to c} p(t) = p(c)$, equivalently, $\lim_{t \to 0} (p(c+t) - p(c)) = 0$. Since $p(c) \neq 0$, $|p(c)|/2 > 0$. Then for $\eeps = |p(c)|/2 > 0$, there exists some $\delta > 0$ such that
	$$0 < |t| < \delta \quad \To \quad |p(c+t) - p(c)| < \frac{|p(c)|}{2}.$$
	Expanding the inequality,
	$$p(c) - \frac{|p(c)|}{2}< p(c+t) < p(c) + \frac{|p(c)|}{2}.$$
	We consider the cases $p(c) > 0$ and $p(c) < 0$.
	\begin{enumerate}[label=\qrom]
		\item If $p(c) > 0$, then the first inequality yields
		$$p(c+t) > p(c) - \frac{|p(c)|}{2} = |p(c)| - \frac{|p(c)|}{2} = \frac{|p(c)|}{2} =: r > 0,$$
		which gives
		$|p(c+t)| = p(c+t) > r$.
		\item If $p(c) < 0$, then the second inequality yields
		$$p(c+t) < p(c) + \frac{|p(c)|}{2} = -|p(c)| + \frac{p(c)}{2} = -\frac{|p(c)|}{2} =: -r < 0,$$
		which gives
		$|p(c+t)| = -p(c+t) > -(-r) = r$.
	\end{enumerate}
\end{proof}
\begin{proof}[Proof of Theorem~\ref{thm: rational_functions}]
	Equivalently, we want to establish that
	$$\lim_{t \to 0} \paren{\dsfrac{p(c+t)}{q(c+t)} - \dsfrac{p(c)}{q(c)}} = 0.$$
	By algebra,
	\begin{align*}
		\parenl{\dsfrac{p(c+t)}{q(c+t)} - \dsfrac{p(c)}{q(c)}}
		&= \parenl{\frac{q(c)p(c+t) - p(c)q(c+t)}{q(c)q(c+t)}}\\
		&= \parenl{\frac{q(c)p(c+t) - q(c)p(c) + q(c)p(c) - p(c)q(c+t)}{q(c)q(c+t)}}\\
		&\leq \frac{|q(c)||p(c+t)-p(c)|}{|q(c)||q(c+t)|} + \frac{|p(c)||q(c+t)-q(c)|}{|q(c)||q(c+t)|}\\
		&\leq \frac{1}{|q(c+t)|} \cdot |p(c+t)-p(c)| + \frac{|p(c)|}{|q(c)|} \cdot \frac{1}{|q(c+t)|} \cdot |q(c+t)-q(c)|.
	\end{align*}
	Fix $\eeps > 0$. Choose $\delta = \square > 0$ to be determined. Suppose $0 < |t| < \delta$. The key is to stipulate $\square \leq \delta_1$ so that
	$$\dsfrac{1}{|q(c+t)|} > r > 0$$ by Lemma~\ref{thm: continuity_positivity}. In practice, trial-and-error is sufficient to determine $\delta_1$ (see Example~\ref{eg: rational_limit}). This allows us to bound
	\begin{align*}
	\parenl{\dsfrac{p(c+t)}{q(c+t)} - \dsfrac{p(c)}{q(c)}}
	&\leq \frac{1}{|q(c+t)|} \cdot |p(c+t)-p(c)| + \frac{|p(c)|}{|q(c)|} \cdot \frac{1}{|q(c+t)|} \cdot |q(c+t)-q(c)|\\
	&\leq \frac 1r \cdot \paren{|p(c+t)-p(c)| + \frac{|p(c)|}{|q(c)|}|q(c+t)-q(c)|}
	\end{align*}
	Since $p$ is a polynomial, we can stipulate $\square \leq \delta_2$ as per Theorem~\ref{thm: polynomial_cty} so that 
	$$0 < |t| < \delta_2 \quad \To \quad |p(c+t)-p(c)| < \frac{r}{2} \cdot \eeps.$$
	Since $q$ is a polynomial, if $p(c) \neq 0$, we can stipulate $\square \leq \delta_3$ as per Theorem~\ref{thm: polynomial_cty} so that 
	$$0 < |t| < \delta_3 \quad \To \quad |q(c+t)-q(c)| < \frac{r \cdot |q(c)|}{2\cdot |p(c)|} \cdot \eeps.$$
	Thus, in the case $p(c) \neq 0$, we let $\square = \min\{\delta_1,\delta_2,\delta_3\}$, where each of the $\delta_i$'s can be computed fairly easily, so that
	\begin{align*}
	\parenl{\dsfrac{p(c+t)}{q(c+t)} - \dsfrac{p(c)}{q(c)}}
	&\leq \frac 1r \cdot \paren{|p(c+t)-p(c)| + \frac{|p(c)|}{|q(c)|}|q(c+t)-q(c)|}\\
	&< \frac 1r \cdot \paren{\frac r2 \cdot \eeps + \frac{|p(c)|}{|q(c)|} \cdot \frac{r \cdot |q(c)|}{2\cdot |p(c)|} \cdot \eeps} = \eeps.
	\end{align*}
\end{proof}

\end{document}